\documentclass[12pt, a4paper]{article}

\usepackage[T1]{fontenc}
\usepackage{amsmath}
\usepackage{amsthm}
\usepackage{amsfonts}
\usepackage{graphicx}
\usepackage{epstopdf} 
\usepackage{authblk} 
\usepackage{mathtools} 
\usepackage{bbm}
\usepackage{hyperref}

\theoremstyle{plain}
\newtheorem{thm}{Theorem}
\newtheorem{lem}[thm]{Lemma}
\newtheorem{cor}[thm]{Corollary}
\theoremstyle{definition}
\newtheorem{rem}[thm]{Remark}

\renewenvironment{proof}[1][]{\smallskip \noindent {\bf #1:}\hskip\labelsep}{\hspace*{\fill}$\square$\medskip\par} 
\let\bfseriesmakro\bfseries\def\bfseries{\bfseriesmakro\mathversion{bold}} 

\def\P{\mathbb{P}}
\def\E{\mathbb{E}}
\def\eps{{\varepsilon}}
\def\diff{\,\mathrm{d}}

\hypersetup{
    pdftitle={Improving the performance of polling models using forced idle times},
    pdfauthor={Aurzada and Schwinn},
    colorlinks=true,
    linkcolor=blue,
    citecolor=red,
}

\begin{document}

\date{\today}
\title{Improving the performance of polling models using forced idle times}
\author[1,2]{Frank Aurzada}
\author[2]{Sebastian Schwinn}
\affil[1]{Department of Mathematics, Technische Universit\"at Darmstadt}
\affil[2]{Graduate School CE, Technische Universit\"at Darmstadt}
\maketitle

\begin{abstract}
We consider polling models in the sense of Takagi~\cite{takagi}. 
In our case, the feature of the server is that it may be forced to wait idly for new messages at an empty queue instead of switching to the next station.
We propose four different wait-and-see strategies that govern these waiting periods.
We assume Poisson arrivals for new messages and allow general service and switchover time distributions.
The results are formulas for the mean average queueing delay and characterisations of the cases where the wait-and-see strategies yield a lower delay compared to the exhaustive strategy.
\end{abstract}


\bigskip


\section{Introduction}
\subsection{Model}
We investigate a polling model in the sense of~\cite{takagi} consisting of $N \geq 1$ stations which are served by one server.
The stations are labelled by the indices from~$1$ to~$N$ and served in ascending, cyclic order.

Each station~$i$ has its own queue which is fed by messages generated by a Poisson arrival process with intensity $\lambda_i$.
Each message has a random length (also called service time).
The mean and second moment of the message length distribution are assumed to be finite and denoted by $b_i$ and $b_i^{(2)}$, respectively.

Switching between stations takes a non-negative random idle time, called switchover time, where the server does not serve any messages at any station.
The random switchover time~$R_i$ from station~$i$ to the next station (with distribution function~$F_{R_i}$) is assumed to have finite mean~$r_i$ and finite second moment~$r_i^{(2)}$.
We consider both non-deterministic and deterministic switchover times (in the latter case $r_i^{(2)}=r_i^2$ for $i=1,\dots,N$).
The sum of the mean switchover times is denoted by $r_0 \coloneqq \sum_{i=1}^N r_i$ and the second moment of the sum of all switchover times by $r_0^{(2)} \coloneqq \sum_{i=1}^N r_i^{(2)} + \sum_{i,j=1,i\neq j}^N r_i r_j$.

The message generation process, the lengths of the messages, and the switchover times are assumed to be independent (everything among each other and with respect to the other processes and stations).

The goal is to obtain explicit formulas for the mean average queueing delay of a message in a polling model with a given wait-and-see strategy in steady state.
The delay is the time a message experiences from the point in time when it arrives in one of the queues until its service starts, i.e., excluding the service time.
The expected delay of a message generated at station~$i$ is denoted by $\E D_i$.
The mean average queueing delay is then defined by
$$
\bar{D} \coloneqq \sum_{i=1}^N \frac{\rho_i}{\rho_0}\, \E D_i,
$$
where $\rho_i \coloneqq \lambda_i b_i$ is the traffic load at station~$i$ and $\rho_0 \coloneqq \sum_{i=1}^N \rho_i$ is the total load offered to the system.
We stress that the delays of the different stations are weighted by the traffic intensity $\rho_i$, which implicitly includes weighting by the mean message lengths, whereas the delays $\E D_i$ do not include weighting the delay of the individual messages with their lengths.
The mean average queueing delay, which we often just abbreviate as \emph{delay}, is called \emph{intensity weighted mean waiting time} by Takagi~\cite{takagi}.

\subsection{Wait-and-see strategies}
First, we describe the behaviour of the server in general:
The server arrives at a station and starts serving in an exhaustive fashion, i.e., serving all waiting messages and newly arriving messages (first come, first served) until the queue is empty.
However, once the station is empty or if the server finds an empty station upon its arrival, the server may not immediately switch to the next station; it rather turns idle for some time in order to wait for possibly newly arriving messages (`wait-and-see').
As soon as a new message arrives, the server starts serving immediately and in an exhaustive fashion.
Once finished, the server may again turn idle and wait for new messages.

For each of the four strategies considered here, the behaviour of the server at station~$i$  is governed by a fixed, real parameter~$T_i \geq 0$ which has different interpretations (see below).
Of course, the server is not allowed to be idle if at its present station messages are waiting to be served.
The reason for waiting depends only on the current station in the current cycle, i.e., on the evolution of the traffic at the present station since the server arrived there.
The server must not use any information about the current queue status at other stations nor about the future of the arrival process at any station.
If $T_i=0$ holds, the service discipline is exhaustive at station~$i$ and there is no state of `wait-and-see' at station~$i$.
If this is the case for all stations, we call it the \emph{exhaustive strategy}.

Now, we specify the four different wait-and-see strategies.
Strategy~I is extensively analysed by Aurzada et al.\ \cite{aurzada} and Strategy~IV is examined by Boxma et al.\ \cite{boxma2002} for $N=2$ stations and $T_2=0$.
As far as we know, there are no results in the literature on Strategy~II and~III.
\begin{itemize}
 \item Under \textbf{Strategy~I}, the server has to wait idly the total time $T_i$ for new messages at station~$i$ per cycle.
 Depending on the arrival process, this credit $T_i$ is spent altogether in one single period or in some periods interleaved by different busy periods.
 \item \textbf{Strategy~II} is defined as follows:
 The server has to stay at least the minimum sojourn time $T_i$ at station~$i$ per cycle.
 We can regard it as a timer starting upon arrival of the server at this station.
 Once the server has spent the minimum sojourn time at the station (possibly consisting of several busy and waiting periods), the server exits the station if the queue is empty.
 However, if there are still messages waiting or in service as the timer runs out, the server continues serving in an exhaustive fashion and switches to the next station as soon as the queue is empty.
 \item \textbf{Strategy~III} is a modification of the previous one. 
 Here, the server is forced to stay at least the fixed time $T_i$ at station~$i$ after becoming idle for the first time at this station in this cycle.
 If there are no messages waiting upon arrival of the server, the timer starts immediately as in the case for Strategy~II.
 Otherwise, the timer starts running just after the first busy period.
 \item \textbf{Strategy~IV} is also similar to Strategy~II.
 However, the timer is only activated if the server finds station~$i$ empty upon arrival.
 In this case, the server remains dormant for at most the time $T_i$, waiting for the first arriving message.
 If the timer expires before the first arrival occurs, the server switches to the next station.
 On the other hand, if a new message arrives before the timer expires, the server starts serving immediately and in an exhaustive fashion.
 After this busy period, the server does not wait any longer at this station in the current cycle and switches to the next station.
\end{itemize}

We stress that we only deal with strategies where the wait-and-see timers are deterministic.
In order to yield a lower minimal delay, we conjecture that deterministic timers do a better job than random timers.
Simulations have indicated that such an additional randomness (of the timer) in the polling model has no positive effect on the minimal delay.

\subsection{Overview of contents}
The results of this paper are as follows:
\begin{itemize}
 \item We give a formula for the mean average queueing delay in a polling model with $N$ stations and Strategy~III (Theorem~\ref{thm:main}).
 \item We prove a formula for the mean average queueing delay in a polling model with $N=2$ stations and Strategy~II (Theorem~\ref{thm:main_2}).
 \item We extend~\cite{boxma2002} to timers at \emph{both} stations and give a formula for the mean average queueing delay for Strategy~IV (Theorem~\ref{thm:main_2}).
\item We characterise the cases for a polling model with $N=2$ stations where these strategies yield a lower delay compared to the exhaustive strategy (Theorems~\ref{thm:cond_2} and~\ref{thm:cond_2_sym}).
\end{itemize}

The remainder of this paper is structured in the following way:
In Section~\ref{subsec:rel_work}, we outline related work.
Section~\ref{sec:results} contains the formulas for the mean average queueing delay (Section~\ref{subsec:basic}) and the cases where it is worth waiting (Section~\ref{subsec:is_it_worth}).
All proofs of the results are collected in Section~\ref{sec:proofs}.

\subsection{Related work}\label{subsec:rel_work}
Aurzada et al.\ \cite{aurzada} analyse Strategy~I and give an explicit formula for the mean average queueing delay in a polling model with $N$ stations.
They characterise several cases where Strategy~I yields a lower delay compared to the exhaustive strategy.
In these cases, the optimal parameters $T_i$ can be computed explicitly.
Finally, they give a lower bound for the delay for a class of wait-and-see strategies which includes Strategy~\mbox{I--IV}.

In~\cite{boxma2002}, Boxma et al.\ focus on a two-queue polling model with a timer as in Strategy~IV at station~$1$ which may be random.
They examine different configurations: Either both stations are served exhaustively, or one station is controlled by the $1$-limited protocol whereas the other station is served in an exhaustive fashion.
The main results are the probability generating function of the queue lengths, expressions for pseudo-conservation laws, and the Laplace transform of the stationary waiting times.

Besides the main references~\cite{aurzada} and~\cite{boxma2002}, further papers deal with service strategies which have in common that the server does not necessarily switch to the next station when the current queue is empty.
Polling models with deterministic sojourn times and preemptive service are considered in~\cite{xie} and with exponentially distributed sojourn times in~\cite{dehaan}.
Similar to Strategy~IV, in the setting of~\cite{afanassieva} the server waits exactly for the first arriving message at an empty station.
In~\cite{cooper,pekoz,samaddar}, forced idle times are examined where the server is not allowed to resume service immediately as soon as a new message arrives during these idle periods.

Furthermore, there are several works that investigate polling models with time-limited service.
There, messages are served at a station for a certain period of time or until the queue is empty, whichever occurs first.
If there is still work at the station when the timer expires, the server either finishes all the present work, or completes only the service of the currently served message, or stops working immediately at this station and switches to the next station.
We refer to~\cite{alhanbali,dehaanphd,yechiali,leung,li} for random time limits (in particular exponentially distributed timers).
In~\cite{souza} and~\cite{frigui}, deterministic time limits are studied.

\section{Results}\label{sec:results}
In this section, we give formulas for the the mean average queueing delay and characterise the cases for the wait-and-see strategies where it is favourable (in the sense of a lower delay) to possibly wait at a station instead of switching.
From now on, we assume that the stability condition $\rho_0 < 1$ of the polling model holds.

\subsection{Basic theorems}\label{subsec:basic}
Theorems~\ref{thm:main} and~\ref{thm:main_2} provide formulas for the mean average queueing delay in terms
\begin{itemize}
 \item of the system parameters $\lambda_i$, $b_i$, $b_i^{(2)}$, $r_i$, $r_i^{(2)}$ for $i=1,\dots,N$, and
 \item of the parameter-dependent quantity $\textbf{S} \coloneqq \left(f_i, w_i, \tilde{r}_i \right)_{i=1,\dots,N}$ of expectations in steady state which are defined in the next paragraph
 and which vary depending on the wait-and-see strategy including the parameters~$T_i$.
Specifying these expectations for Strategy~\mbox{II--IV} in Section~\ref{subsec:deter} is the main novelty in this paper.
\end{itemize}

We define the expected time per cycle which the server waits at station~$i$ by~$f_i$.
We use $f_0 \coloneqq \sum_{i=1}^N f_i$ for the total expected waiting time of the server per cycle (i.e., idle times without switchover times).
The expected backward recurrence time (expected spent time) $w_i$ is defined by the expectation of the elapsed time since arriving at station~$i$ at a random point in time while waiting at station~$i$.
Furthermore, we introduce the conditional mean switchover time~$\tilde{r}_i$ from station~$i$ to the next station:
Given a random point in time while waiting at this next station, $\tilde{r}_i$ is the expected length of the preceding switchover time.

\begin{thm}\label{thm:main}
The mean average queueing delay of a message in a polling model with Strategy~III is given by
\begin{equation*}
\begin{split}
 \bar{D} = & \; \frac{\sum_{i=1}^N \lambda_{i} b_{i}^{(2)}}{2(1-\rho_{0})} + \frac{(r_0+f_0) \left(\rho_{0}^{2} -\sum_{i=1}^N \rho_{i}^{2} \right)}{2 \rho_{0} (1-\rho_0)} + \frac{\frac{1}{2} \rho_{0}r_{0}^{(2)}+r_{0} \sum_{i=1}^N f_{i}(\rho_{0} - \rho_{i})}{\rho_{0} (r_0+f_0)} \\
 & + \frac{1}{\rho_{0}(r_0+f_0)} \left[ \sum_{i=1}^N f_{i} w_i (\rho_{0} - \rho_{i}) + \sum_{1 \leq i < j \leq N} f_i f_j (\rho_0 - \rho_i - \rho_j) \right] \\
 & - \frac{\sum_{i=1}^N  f_i \rho_i(\rho_0 - \rho_i)}{\rho_{0} (1- \rho_{0})}. 
\end{split}
\end{equation*}
\end{thm}

We refer to Aurzada et al.\ \cite{aurzada} for the delay of a message in a polling model with Strategy~I. 
For Strategy~II and~IV, we restrict the number of stations to $N=2$ due to the technical effort 
that would be required otherwise to compute further parameter-dependent quantities which would arise in the formula for the delay.

\begin{thm}\label{thm:main_2}
The mean average queueing delay of a message in a polling model with $N=2$ stations and Strategy~\mbox{II--IV} is given by
\begin{equation}\label{eq:D}
\begin{split}
 \bar{D} = & \; \frac{\sum_{i=1}^2 \lambda_i b_i^{(2)}}{2(1- \rho_0)} + \frac{r_0 \rho_1 \rho_2}{\rho_0 (1- \rho_0)} + \frac{r_0^{(2)}}{2(r_0+f_0)} \\
 & + \frac{\rho_2 f_1}{\rho_0 (r_0+f_0)} (r_1 + \tilde{r}_2 + w_1) \\
 & + \frac{\rho_1 f_2}{\rho_0 (r_0+f_0)} (\tilde{r}_1 + r_2 + w_2). 
\end{split}
\end{equation}
\end{thm}

Both Theorems~\ref{thm:main} and~\ref{thm:main_2} are valid for general distributions of the service times.
However, we emphasise that for Strategy~II and~III we are only able to compute the quantity~$\textbf{S} = \left(f_i, w_i, \tilde{r}_i \right)_i$ explicitly for exponentially distributed service times because formula~\eqref{eq:P_jk} below is only available for the M/M/1 queue in the literature, for instance.
The computation is specified in Section~\ref{subsec:deter}.

\begin{figure}[ht]
\includegraphics[width=1\linewidth]{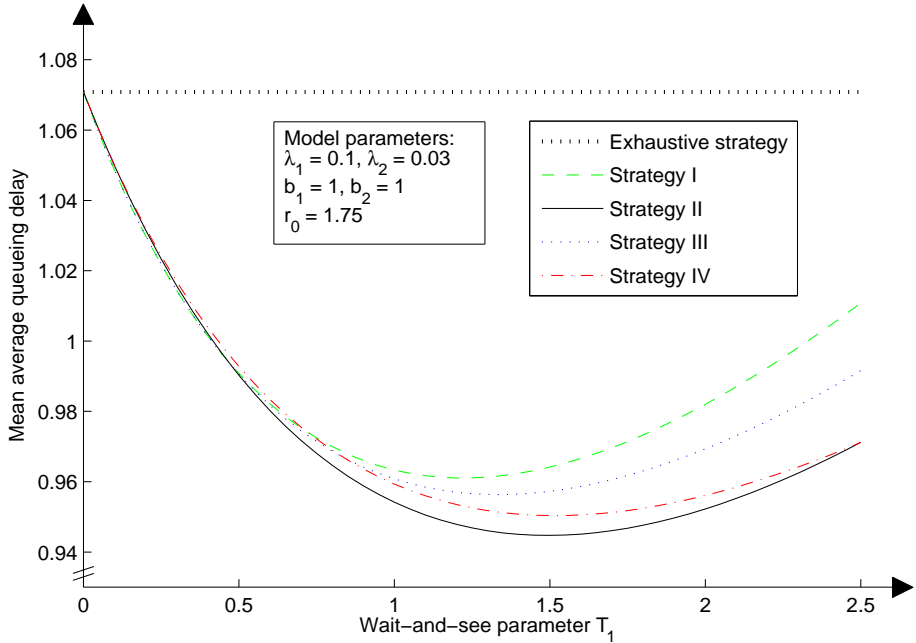}  
\caption{Comparison of the delays for the strategies vs.\ the wait-and-see parameter~$T_1$.}\label{fig:delay_all}
\end{figure}

Figure~\ref{fig:delay_all} provides a typical relation between the delays for all four wait-and-see strategies.
We consider a polling model with $N=2$ stations where the server is not allowed to wait at station~$2$.
The switchover times are deterministic, symmetrically split among the switchovers and the service times are exponentially distributed.
The delay for Strategy~I was obtained by the formula from~\cite{aurzada}, and we used Theorem~\ref{thm:main_2} and the value for $\textbf{S} = \left( f_i, w_i, \tilde{r}_i \right)_i$ from Section~\ref{subsec:deter} to compute the data for Strategy~\mbox{II--IV}.

The ranking of the wait-and-see strategies with respect to the minimal delay observed in Figure~\ref{fig:delay_all} can be explained naturally:
In the best case, the server exits the current station as soon as there is enough work waiting at the other station.
Since the server does not have any information about the queue status at the other station, the sojourn time at the current station is the crucial quantity in order to estimate the workload generated at the other station.
Hence, there is an optimal sojourn time for each station, and the minimal delay is attained if the expected sojourn time agrees with the optimal sojourn time best with a small variance.

Therefore, we conjecture that Strategy~III always yields a lower minimal delay compared to Strategy~I and that Strategy~II is the best of the investigated wait-and-see strategies. 

\subsection{Is it worth waiting?}\label{subsec:is_it_worth}
Theorem~\ref{thm:main_2} allows us to put the following question:
Given the system parameters, how does one have to adjust the parameters $T_i\geq 0$ such that the delay is minimised.
We can not compute a minimiser of this problem
$$
\min_{T_1 \geq 0, \, T_2 \geq 0} \bar{D}(T_1,T_2) 
$$
for Strategy~\mbox{II--IV} analytically.
Nevertheless, we do obtain necessary and sufficient conditions for these wait-and-see strategies in a polling model with exponentially distributed service times such that it is favourable to wait in comparison to the exhaustive strategy.
As a summary one can say that the benefit of waiting arises from the asymmetry of the system or from non-deterministic switchover times.

We say that `it is worth waiting (at station $i$)' if there is a $T_i>0$ such that the delay is lower than for the exhaustive strategy, i.e., $\bar{D}(T_1,T_2) < \bar{D}(0,0)$. 
Note that we only consider the two cases with the additional restriction $T_2=0$ and $T_1=T_2$, respectively.

\begin{thm}\label{thm:cond_2}
Let $T_2=0$.
It is worth waiting at station~$1$ in a polling model with $N=2$ stations and
\begin{itemize}
\item Strategy~III if and only if 
$$
\frac{r_0^{(2)}}{2r_0^2} - \frac{\rho_2 (1- \rho_2)}{\rho_0 (1- \rho_0)} > 0,
$$
\item Strategy~II as well as Strategy~IV if and only if 
\begin{equation}\label{eq:cond_IV}
\frac{r_0^{(2)}}{2r_0 \left(r_1 + \tilde{r}^{\rm{IV}}_2 \right)} - \frac{\rho_2}{\rho_0} > 0,
\end{equation}
where the quantity $\tilde{r}^{\rm{IV}}_2$ is given by~\eqref{eq:r_IV} below.
In the case of a deterministic switchover time $R_2$, inequality~\eqref{eq:cond_IV} simplifies to $\rho_1 > \rho_2$.
\end{itemize}
\end{thm}

\begin{rem}
We observe that for a sufficiently large relative variance of the sum of the switchover times $ \frac{r_0^{(2)}-r_0^2}{r_0^2}$ it is even worth waiting at station~$1$ in spite of a lower traffic load $\rho_1 < \rho_2$.
As a consequence of this, it is favourable to have \emph{positive} parameters $T_i$ at both stations instead of just allowing `wait-and-see' at the station with higher traffic load.
\end{rem}

Similar to above, we get necessary and sufficient conditions for a symmetric polling model with $\rho_1=\rho_2$ such that it is worth waiting with the restriction $T_1=T_2$.
The arrival rates, message length and switchover time distributions are also assumed to be the same for both stations for Strategy~II and~IV but we can omit this requirement for Strategy~III.

\begin{thm}\label{thm:cond_2_sym}
It is worth waiting with $T_1=T_2$ in a symmetric polling model with $N=2$ stations and
\begin{itemize}
 \item Strategy~III if and only if
$$
\frac{r_0^{(2)}}{r_0^2} - \frac{1-\rho_1}{1-\rho_0} > 0
$$
(that can only be satisfied for non-deterministic switchover times),
\item Strategy~II as well as Strategy~IV if and only if the switchover times are non-deterministic.
\end{itemize}
\end{thm}

We give a direct consequence of the two preceding theorems:

\begin{cor}\label{cor:conseq}
There are parameter settings of a polling model with $N=2$ stations where Strategy~II and~IV yield a lower delay than Strategy~I and~III, i.e., it is only worth waiting with Strategy~II and~IV.
\end{cor}

\section{Proofs}\label{sec:proofs}
\subsection{Proofs of the basic theorems}
We show how to derive Theorems~\ref{thm:main} and~\ref{thm:main_2} which are based on a decomposition principle from~\cite{boxma1987} and on the technique of the proofs of Theorems~$1$ and~$8$ from~\cite{aurzada}.
We mention that the proofs of Theorems~\ref{thm:main} and~\ref{thm:main_2} are quite standard and that the key novelty is the computation of the parameter-dependent quantities in Section~\ref{subsec:deter}.

\begin{proof}[Proof of Theorem~\ref{thm:main}]
First, we recall some important identities: 
The cycle time is the time that the server takes from its arrival at station~$1$ to the next arrival at this station.
The mean cycle time in steady state is denoted by~$\E C$ and is given by
$$
\E C = \frac{r_0+f_0}{1-\rho_0}.
$$
Indeed, this can be argued by looking at the expected time which the server is idle per cycle. 
This expectation equals the sum of all mean switchover and waiting times, i.e., $r_0+f_0$.
On the other hand, we can represent this expected idle time using the total load offered to the system which leads to~$(1-\rho_0)\E C$.

Next, we refer to the workload decomposition in~\cite{boxma1987} and~\cite{boxma2002}.
As a consequence of this decomposition principle, we obtain
\begin{equation}\label{eq:work_decomp}
\E V = \E V^{\rm{M/G/1}} +q \E V^{\rm{switching}} + (1-q) \E V^{\rm{waiting}},
\end{equation}
where $q \coloneqq \P (\text{server is switching} \mid \text{server is idle})$ and $V$ is the workload at a random point in time in steady state.
The workload consists of the sum of all message lengths that are present in the system including the remaining service time of the currently served message.
The quantities $V^{\rm{M/G/1}}$ and $V^{\rm{switching}}$ ($V^{\rm{waiting}}$) refer to the workload in the same polling model without switchover and waiting times, and to the workload given that the server is switching (waiting) at a random point in time, respectively.
Furthermore, we can determine the expected workload differently by
$$
\E V =  \sum_{i=1}^N b_{i} \E[\text{number of messages in queue at station~$i$}] + \sum_{i=1}^N \rho_{i}\frac{b_{i}^{(2)}}{2b_{i}},
$$
where the quotient is the expected residual service time of a currently served message at station~$i$.
Using Little's law, this equation can be rearranged into
\begin{equation}\label{eq:work}
\E V = \rho_0 \bar{D} + \sum_{i=1}^N \rho_{i}\frac{b_{i}^{(2)}}{2b_{i}}.
\end{equation}
Therefore, we can combine~\eqref{eq:work_decomp} and~\eqref{eq:work} in order to obtain a representation of the delay $\bar{D}$.
The quantity $\E V^{\rm{M/G/1}}$ is given in the literature, e.g., in~\cite[p.~206]{asmussen}.

From now on, we focus on the expected workload present while switching~$\E V^{\rm{switching}}$ and waiting~$\E V^{\rm{waiting}}$.
The former does not directly depend on the given wait-and-see strategy so that we can proceed in the same way as in~\cite{aurzada}.
On the other hand, the particular wait-and-see strategy influences the expected workload present while waiting.
It remains to give the general formula for $\E V_i^{\rm{waiting}}$, the expected workload that is present in the system at a random point in time when the server is waiting at station~$i$.
Following the computation in~\cite{aurzada} (see equation $(23)$ there), we obtain
\begin{equation*}
\begin{split}
 \E V_i^{\rm{waiting}} = & \; \sum_{j<i} r_{j} \left(\sum_{l=i+1}^N\rho_{l} + \sum_{l=1}^j \rho_{l}\right) + \sum_{j>i}r_{j} \sum_{l=i+1}^{j} \rho_{l} \\
  & +\sum_{j<i} \rho_{j} \E C \left(\sum_{l=i+1}^N \rho_{l} + \sum_{l=1}^{j-1} \rho_{l}\right) + \sum_{j>i}\rho_{j} \E C \sum_{l=i+1}^{j-1} \rho_{l} \\
  & +\sum_{j<i}f_{j} \left(\sum_{l=i+1}^N \rho_{l} + \sum_{l=1}^{j-1} \rho_{l}\right) + \sum_{j>i}f_{j} \sum_{l=i+1}^{j-1} \rho_{l} \\
  & +(\rho_{0}-\rho_{i}) w_i, 
\end{split}
\end{equation*}
where $w_i$ denotes the expectation of the elapsed time since arriving at station~$i$ at a random point in time while waiting at station~$i$.
Combining all the relevant equations, we get the formula in Theorem~\ref{thm:main} for the delay.
\end{proof}

\begin{proof}[Proof of Theorem~\ref{thm:main_2}]
For Strategy~III, we just obtain the formula from Theorem~\ref{thm:main} by replacing $N$ with $2$.
Now, we deal with Strategy~II and~IV in a polling model with $N=2$ stations.
The only part that differs from the proof of Theorem~\ref{thm:main} is the computation of $\E V_i^{\rm{waiting}}$.
We focus on this quantity for $i=1$.
Since the server is currently waiting at station~$1$, the present workload has not been generated at this station.
Therefore, the workload which is currently present can only consist of messages that have been generated at station~$2$ since exiting that station.
The expectation of the elapsed time is the sum of the mean switchover time from station~$2$ to station~$1$ and the expected backward recurrence time $w_1$.
Keeping in mind that the server is at a random point in time while waiting, we have to use $\tilde{r}_2$ instead of $r_2$ and get
$$
\E V_1^{\rm{waiting}} = \rho_2 (\tilde{r}_2 + w_1).
$$
For $\E V_2^{\rm{waiting}}$, we just have to exchange the roles of $1$ and $2$. 
\end{proof}

\begin{rem}\label{rem:r_cond}
Actually, the conditional mean switchover time~$\tilde{r}_i$ only differs from~$r_i$ for a non-deterministic switchover time for Strategy~II and~IV.
In the case of deterministic switchover times for Strategy~II and~IV, or in the case of a polling model with Strategy~I and~III (waiting occurs if $T_i>0$, independently of the switchover times), we have the equality $\tilde{r}_i=r_i$.
\end{rem}

\begin{rem}
We briefly refer to Theorem~$8$ in~\cite{aurzada} which provides a lower bound for the delay for a class of wait-and-see strategies (including Strategy~\mbox{I--IV}).
The bound given there is correct for a polling model with $N=2$ stations and deterministic switchover times.
In the case of $N=2$ and non-deterministic switchover times, we can replace $r_k$ (in $(35)$ there) by $\E[R_k \mid B_0]$ from~\eqref{eq:r_IV} below.
This is due to the fact that non-deterministic switchover times can have an impact on the existence of a waiting period at the next station (cf.\ Remark~\ref{rem:r_cond}).
In addition for a polling model with $N>2$ stations, we have to bound below the expected sojourn times which the server spends at preceding stations given a random point in time while waiting.
\end{rem}

\subsection{\texorpdfstring{Determination of $\textbf{S}=\left(f_i, w_i, \tilde{r}_i \right)_i$}{Determination of S}}\label{subsec:deter}
The general formulas for the delay in Theorems~\ref{thm:main} and~\ref{thm:main_2} require the specification of $\textbf{S} = \left(f_i, w_i, \tilde{r}_i \right)_i$ according to the wait-and-see strategy.
The real novelty of this work is the determination of these quantities in this section.
Note that we restrict the service times of the messages to exponential distributions with parameter $\mu_i \coloneqq \frac{1}{b_i}$ at station~$i$ for $i=1,\dots,N$ for Strategy~II and~III.
After the following preparations, we discuss the different wait-and-see strategies separately where some parts of Strategy~IV come from~\cite{boxma2002}.
For the sake of simplicity, we deal with Strategy~III before Strategy~II.
                                    
\subsubsection{Preparations}
It is helpful to introduce $c_i$ the expected time per cycle in steady state which the server spends at station~$i$.
This expression is directly related to the mean cycle time $\E C$ and to the expected waiting time $f_i$ at station~$i$ by the equation
$$
c_i= \rho_i \E C + f_i.
$$
Moreover, we define $c_0 \coloneqq \sum_{i=1}^N c_i$ and obtain $ \E C = c_0 + r_0.$

We require a time-dependent state probability (denoted by $P_{j,k}(x)$) to analyse the delay for Strategy~II and~III, and we require the distribution of the length of a busy period to analyse the delay for Strategy~II and~IV.

\paragraph{The probability $P_{j,k}(x)$.}
According to~\cite[p.\ 55]{kleinrock}, we denote by $P_{j,k}(x)$ the probability that the queue length of an M/M/1 queue (in the sense of the population size of a birth-death process) with arrival rate~$\lambda_i$ and service rate~$\mu_i$ is $k$ at time $x$ given that the queue length is $j$ at time zero.
We introduce the abbreviation $a_i \coloneqq 2\mu_i \sqrt{\rho_i}$, where the traffic load $\rho_i$ equals $\frac{\lambda_i}{\mu_i}$, and the modified Bessel functions $I_k(x)$ of the first kind of order~$k$, which can be defined by
$$
I_k(x) \coloneqq \sum_{m=0}^{\infty} \frac{\left(\frac{x}{2}\right)^{k+2m}}{(k+m)!m!} \quad \text{for } k \in \mathbb{N}_0
$$
and $I_{-k}(x) \coloneqq I_k(x)$ for $k \in \mathbb{N}$. 
Finally, we have
\begin{equation}\label{eq:P_jk}
\begin{split}
 P_{j,k}(x) = e^{-(\lambda_i + \mu_i)x} \Bigg[ \rho_i^{\frac{k-j}{2}} I_{k-j}(a_ix) & + \rho_i^{\frac{k-j-1}{2}} I_{k+j+1}(a_ix) \\
 & + (1-\rho_i) \rho_i^{k} \sum_{l=k+j+2}^{\infty} \rho_i^{-\frac{l}{2}} I_l(a_ix) \Bigg] \\
\end{split} 
\end{equation}
due to~\cite[p.\ 77]{kleinrock}.
We emphasise that the probability $P_{j,k}(x)$ differs depending on~$i$ but we omit such an additional index because it arises out of the context.

\paragraph{The density $g_i$ of a busy period.}
The density of the length of a busy period at station~$i$ is denoted by~$g_i$ and the $n$-fold convolution of~$g_i$ with itself by~$g_i^{(*n)}$.
We get
$$
g_i(x)= \sum_{n=1}^\infty e^{-\lambda_i x} \frac{(\lambda_i x)^{n-1}}{n!} b_i^{(*n)}(x) \quad \text{for } x \geq 0
$$
from~\cite[p.\ 226]{kleinrock}.
Note that with abuse of notation $g_i^{(*0)}$ represents the Dirac delta function according to the property that the length of $0$ busy periods is zero.
The density $b_i^{(*n)}$ is the $n$-fold convolution of the service time with itself.
For exponentially distributed service times, we obtain the density
$$
b_i^{(*n)}(x) = \frac{\mu_i^n x^{n-1}}{(n-1)!} e^{-\mu_i x} \quad \text{for } x \geq 0,
$$
of the Erlang($n,\mu_i$) distribution which can also be identified as a gamma distribution.
In this particular case, a further representation of $g_i$ using the modified Bessel function of the first kind of order one is given in~\cite[p.\ 215]{kleinrock}.

\subsubsection{Strategy III}
We denote by $q_i(x)$ the expected number of messages (including the possibly currently served message) present at station~$i$ after time~$x$ given that there is no message present at time zero.
With the probability $P_{0,k}(x)$ which we have just introduced, we get
$$
q_i(x) = \sum_{k=0}^\infty k P_{0,k}(x).
$$
Since we only require the expected number of messages at time $T_i$, we define the short version $q_i \coloneqq q_i(T_i)$.

\paragraph{The expected sojourn time $c_i$.} 
For each station we get the equation
\begin{equation}\label{eq:c_i_III}
c_i= \lambda_i (r_0 + c_0 - c_i) \frac{b_i}{1-\rho_i} + T_i + q_i \frac{b_i}{1-\rho_i} 
\end{equation}
which can be seen as follows:
First of all, the time which the server spends at station~$i$ depends on the elapsed time since exiting this station in the preceding cycle up to the current arrival at this station.
This expected intervisit time of the server at station~$i$ is
$$\E C - c_i = r_0 + c_0 - c_i$$
and the quotient $\frac{b_i}{1-\rho_i}$ is the expected length of a busy period (which is caused by one arriving message).
In order to obtain this latter quantity, we refer to the short calculation using Laplace transforms in~\cite[pp.\ 211--213]{kleinrock}.
Together with the arrival rate $\lambda_i$, we can compute the expected length of the first busy period (generated by the waiting messages) at station~$i$ and get
\begin{equation}\label{eq:first_busy_period}
 \lambda_i (r_0 + c_0 - c_i) \frac{b_i}{1-\rho_i}.
\end{equation}
After the first busy period, the server has to spend the time $T_i$ at this station (which can consist of several busy and waiting periods).
Then, the server exits the station if the queue is empty at time $T_i$.
Alternatively, if there are messages present at time $T_i$, the server continues serving messages until the queue is empty.
This additional time depends on the expected number $ q_i $ of present messages and equals $q_i \frac{b_i}{1-\rho_i}$ in expectation.

Using~\eqref{eq:c_i_III}, we can set up linear system of equations with variables $c_i$.
For instance in the case of two stations, we obtain
$$
c_1 = \frac{r_0 \rho_1 + (1-\rho_1)(1-\rho_2) T_1 + \rho_1 (1-\rho_2) T_2 + (1-\rho_2) q_1 b_1 + \rho_1 q_2 b_2}{1-\rho_0},
$$
\begin{equation}\label{eq:c2_III}
c_2 = \frac{r_0 \rho_2 + (1-\rho_1)(1-\rho_2) T_2 + \rho_2 (1-\rho_1) T_1 + (1-\rho_1) q_2 b_2 + \rho_2 q_1 b_1}{1-\rho_0}.
\end{equation}

\paragraph{The expected backward recurrence time $w_i$.} 
The expectation $w_i$ is the sum of two terms:
On the one hand, there is the expected length of the first busy period at station~$i$ (see term~\eqref{eq:first_busy_period}).
The second summand is the expectation of the elapsed time since becoming idle at station~$i$ for the first time at a random point in time while waiting at this station.
Therefore, we get
\begin{equation}\label{eq:wi_III}
w_i = \lambda_i (r_0 + c_0 - c_i) \frac{b_i}{1-\rho_i} + \frac{\int_0^{T_i} x P_{0,0}(x) \diff x}{\int_0^{T_i} P_{0,0}(x) \diff x},
\end{equation}
where a random point in time while waiting has the density
$$
\frac{P_{0,0}(x)}{\int_0^{T_i} P_{0,0}(y) \diff y} \quad \text{for } x \in [0,T_i].
$$

\paragraph{The conditional mean switchover time $\tilde{r}_i$.}
We have $\tilde{r}_i=r_i$ because there is a waiting period at station~$i$ every cycle for~$T_i>0$ due to the definition of Strategy~III.

\subsubsection{Strategy II}
We focus on the steady-state probabilities~$\pi_n^{(i)}$ for all $n \in \mathbb{N}_0$ that the server finds $n$ messages waiting upon arrival at station~$i$.
We consider deterministic switchover times in this paragraph first.
The following system of equations describes the relation of consecutive visits at the stations.
The probability of finding $n$ messages upon arrival at station~$1$ depends on the intervisit time of the server, i.e., the time since exiting this station in the preceding cycle.
The intervisit time can be divided into the sum of the switchover times and the time which the server spends at station~$2$ between two consecutive visits at station~$1$.
This latter time can be split in two parts:
First, the server stays the minimum sojourn time $T_2$.
The second part consists of the time which the server takes to serve the possibly remaining messages.
This part depends on the number of messages present at time~$T_2$.
Given that the server finds $k$ messages upon arrival at station~$2$, there are $l$~messages present with probability $P_{k,l}(T_{2})$ after spending the minimum sojourn time. 
Then, the length of the second part has the density~$g_2^{(*l)}$ which denotes the density of the sum of $l$ independent busy periods at station~$2$.
We recall that the arrival process at station~$1$ is a Poisson process with arrival rate $\lambda_1$.
The probability of finding $n$ messages at station~$1$ is given by a Poisson distribution with parameter $\lambda_1 t$ if the intervisit time of the server equals $t$.
Therefore, we can conclude the equation
$$
\pi_n^{(1)} = \sum_{k=0}^\infty \pi_k^{(2)} \sum_{l=0}^\infty P_{k,l}(T_{2}) \int_0^{\infty} e^{-\lambda_1 (r_0+T_{2}+x)} \frac{(\lambda_1 (r_0+T_{2}+x))^n}{n!} g_2^{(*l)}(x) \diff x 
$$
for deterministic switchover times. Thereby, we get the coefficients for an infinite linear system of equations 
$\pi^{(1)} = A \pi^{(2)}.$
In the same manner as above, there is a system
$\pi^{(2)} = B \pi^{(1)}.$

If the switchover times are non-deterministic, we can not proceed in such a straightforward way.
Instead, we focus on the queue length distribution at server departure instants.
Note that the queue at departure instants is always empty at the current station.
We denote by $\nu_n^{(i)}$ the steady-state probabilities that there are $n$ messages waiting at the other station upon exit from station~$i$.
Now, we give an explanation for the equation
\begin{equation}\label{eq:nu_n}
\begin{split}
 \nu_n^{(1)} = \sum_{k=0}^\infty & \, \nu_k^{(2)} \sum_{m=0}^\infty \sum_{j=0}^n \Bigg[ \int_0^{\infty} e^{-(\lambda_1+\lambda_2) x} \frac{(\lambda_1 x)^{m}}{m!} \frac{(\lambda_2 x)^{j}}{j!} \diff F_{R_2}(x) \\ 
 & \sum_{l=0}^\infty P_{k+m,l}(T_{1}) \int_0^{\infty} e^{-\lambda_2 (T_{1}+x)} \frac{(\lambda_2 (T_{1}+x))^{n-j}}{(n-j)!} g_1^{(*l)}(x) \diff x \Bigg]
\end{split}
\end{equation}
which consists of similar terms as above.
Given that there are $k$ messages waiting at station~$1$ upon exit from station~$2$, we have $m$ message arrivals at station~$1$ and $j$ message arrivals at station~$2$ while switching to station~$1$.
Therefore, there are $k+m$ messages waiting at station~$1$ upon arrival at this station.
In order to obtain a queue length of $n$ messages at station~$2$ upon exit from station~$1$, a total of $n-j$ messages have to arrive at station~$2$ during this stay.
Then, equation~\eqref{eq:nu_n} follows by considering all possible variations of indices.

From~\eqref{eq:nu_n} and the corresponding observation, we get two systems of equations $\nu^{(1)} = \tilde{A} \nu^{(2)}$ and $\nu^{(2)} = \tilde{B} \nu^{(1)}$ where the coefficients of $\tilde{A}$ are given in~\eqref{eq:nu_n}.
Finally, we are able to determine $\pi_n^{(i)}$ by
\begin{equation}\label{eq:pi_n}
\pi_n^{(1)} = \sum_{k=0}^n \nu_k^{(2)} \int_0^{\infty} e^{-\lambda_1 x} \frac{(\lambda_1 x)^{n-k}}{(n-k)!} \diff F_{R_2}(x).
\end{equation}
For $\pi^{(2)} $, the roles of $1$ and $2$ have to be exchanged.

\paragraph{The expected sojourn time $c_i$.} 
Using the solutions $\pi^{(i)}$, we obtain the expected sojourn time
$$
c_i = T_i + \sum_{k=0}^\infty \pi_k^{(i)} \sum_{l=0}^\infty l P_{k,l}(T_{i}) \frac{b_i}{1-\rho_i}
$$
which the server spends at station~$i$ per cycle.
Here, the series
$$
\sum_{l=0}^\infty l P_{k,l}(T_{i}) 
$$
is the expectation of the number of messages present at station~$i$ after spending the minimum sojourn time $T_i$ given that there are $k$ messages present upon arrival of the server.
The quotient $\frac{b_i}{1-\rho_i}$ is the expected length of a busy period.

\paragraph{The expected backward recurrence time $w_i$.}
In order to determine~$w_i$, we recall the condition that a point in time while the server is waiting is randomly chosen.
We distinguish how many messages are waiting upon arrival of the server at the station.
Therefore, we obtain
$$
w_i = \sum_{k=0}^\infty p_k^{(i)} \frac{\int_0^{T_i} x P_{k,0}(x) \diff x}{\int_0^{T_i} P_{k,0}(x) \diff x},
$$
where $p_k^{(i)}$ denotes the probability of choosing a waiting period during a stay with $k$ messages waiting upon arrival of the server.
Similar to Strategy~III above, the quotient is the expectation of the elapsed time since arriving at station~$i$ at a random point in time while waiting at station~$i$ given that there are $k$ messages waiting upon arrival of the server.

It remains to determine the coefficients $p_k^{(i)}$.
The basic observation is that $p_k^{(i)}$ is proportional to the probability~$\pi_k^{(i)}$ that the server finds $k$ messages waiting upon arrival at station~$i$ and to the expected length of the total waiting time during the stay at such a station, i.e., $\int_0^{T_i} P_{k,0}(x) \diff x$.
Hence, the probability $p_k^{(i)}$ is given by
\begin{equation}\label{eq:p_k}
p_k^{(i)} = \frac{\pi_k^{(i)} \int_0^{T_i} P_{k,0}(x) \diff x}{\sum_{l=0}^\infty \pi_l^{(i)} \int_0^{T_i} P_{l,0}(x) \diff x}.
\end{equation}

\paragraph{The conditional mean switchover time $\tilde{r}_i$.}
If the switchover time from station~$i$ to the next station is deterministic, we get $\tilde{r}_i=r_i$ (cf.\ Remark~\ref{rem:r_cond}).
Otherwise, the conditional mean switchover time $\tilde{r}_i$ from station~$i$ to the next station, given a random point in time while waiting at this next station, can be determined as follows.
We restrict the computation to $i=2$ for the sake of clarity.
First, we introduce the events
\begin{equation*}
\begin{split}
 A_l \coloneqq & \; \{\text{there are $l$ messages waiting at station~$1$ upon exit from station~$2$} \},\\
 B_j \coloneqq & \; \{\text{there are $j$ messages arriving at~$1$ while switching from~$2$ to~$1$} \},\\
 C_k \coloneqq & \; \{\text{there are $k$ messages waiting at station~$1$ upon arrival} \} 
\end{split}
\end{equation*}
for all $j,k,l \in \mathbb{N}_0$.
We distinguish how many messages are waiting upon arrival of the server at station~$1$ just like above.
We get
\begin{equation}\label{eq:r_II}
\tilde{r}_2 = \sum_{k=0}^\infty p_k^{(1)} \, \E[R_2 \mid C_k], 
\end{equation}
where $p_k^{(i)}$ is given by~\eqref{eq:p_k}.
Now, we are left with the specification of the quantity $\E[R_2 \mid C_k]$.
We make use of
$$
C_k = \bigcup_{j=0}^{k} A_{k-j} \cap B_j
$$
and obtain
$$
\E[R_2 \mid C_k] = \sum_{j=0}^k \frac{ \P(A_{k-j} \cap B_j)}{\P(C_k)} \, \E[R_2 \mid A_{k-j} \cap B_j].
$$
Due to the independence of the events $A_{k-j}$ and $B_j$, and the fact that $A_{k-j}$ does not influence the switchover time $R_2$, we get
\begin{equation}\label{eq:r_cond_Ck}
\E[R_2 \mid C_k] = \sum_{j=0}^k \frac{\P(A_{k-j}) \P(B_j)}{\P(C_k)} \, \E[R_2 \mid B_j].
\end{equation}
It remains to determine these quantities.
We can represent event $B_j$ as
\begin{equation}\label{eq:Bj}
B_j= \left\{\sum_{l=1}^j e_l \leq R < \sum_{l=1}^{j+1} e_l \right\},
\end{equation}
where $(e_l)_l$ is a sequence of independent and exponentially distributed random variables with parameter~$1$ which are independent of $R \coloneqq \lambda_1 R_2$ as well.
We get
$$
\lambda_1 \E\left[R_2 \mid B_j \right] = \frac{\E\left[R \, \mathbbm{1}_{B_j}\right]}{\E\left[\mathbbm{1}_{B_j} \right]} = \frac{\E_R\left[R \, \E_{(e_l)_l} \left[\mathbbm{1}_{B_j} \right]\right]}{\E_R \left[\E_{(e_l)_l} \left[\mathbbm{1}_{B_j} \right]\right]}.
$$
We use the property that the sum of independent and identically exponentially distributed random variables is Erlang distributed and thus compute
$$
\E_{(e_l)_l} \left[\mathbbm{1}_{B_j} \right] = \frac{R^j}{j!} e^{-R}.
$$
Therefore, we obtain
\begin{equation}\label{eq:r_cond_Bj}
\E[R_2 \mid B_j] = \frac{\E_R \left[R^{j+1} e^{-R} \right]}{\lambda_1 \E_R \left[R^j e^{-R} \right]} = \frac{\int_0^{\infty} x e^{-\lambda_1 x} \frac{(\lambda_1 x)^j}{j!} \diff F_{R_2}(x)}{\int_0^{\infty} e^{-\lambda_1 x} \frac{(\lambda_1 x)^j}{j!} \diff F_{R_2}(x)}
\end{equation}
and 
$$
\P(B_j) = \E\left[\mathbbm{1}_{B_j} \right] = \int_0^{\infty} e^{-\lambda_1 x} \frac{(\lambda_1 x)^j}{j!} \diff F_{R_2}(x).
$$
Finally, we have
$$
\P(C_k) = \sum_{j=0}^k \P(A_{k-j}) \P(B_j)
$$
due to the independence and $\P(A_{k-j}) = \nu_{k-j}^{(2)}$.

\subsubsection{Strategy IV}
As above, $\pi_n^{(i)}$ is the steady-state probability that the server finds $n$ messages waiting upon arrival at station~$i$.
The method we use to give the characterising system coincides with the method for Strategy~II.
The probability~$\pi_n^{(1)}$ depends on the intervisit time of the server which consists of the switchover times and the time that the server spends at station~$2$ between two consecutive visits at station~$1$.

We have to distinguish whether there is no message or at least one message waiting at station~$2$ because it influences the activation of the timer.
In the first case, either a new message arrives before the timer expires and a busy period starts, or there is no message arrival and the server waits the whole time $T_2$.
For deterministic switchover times, we obtain
\begin{equation*}
\begin{split}
 \pi_n^{(1)} = & \; \pi_0^{(2)} \left[\int_0^{T_2} \int_0^{\infty} e^{-\lambda_1(r_0+x+y)} \frac{(\lambda_1(r_0+x+y))^n}{n!} g_2(x) \diff x \, \lambda_2 e^{-\lambda_2 y} \diff y \right. \\
 & \qquad \qquad \qquad \qquad \qquad \qquad \qquad + \left. e^{-\lambda_1(r_0+T_2)} \frac{(\lambda_1(r_0+T_2))^n}{n!} e^{-\lambda_2 T_2} \right] \\
 & + \sum_{k=1}^\infty \pi_k^{(2)} \int_0^{\infty} e^{-\lambda_1 (r_0+x)} \frac{(\lambda_1 (r_0+x))^n}{n!} g_2^{(*k)}(x) \diff x.
\end{split}
\end{equation*}
Once again, we get systems of equations $\pi^{(1)} = A \pi^{(2)}$ and $\pi^{(2)} = B \pi^{(1)}$.
Note that we are only interested in $\pi_0^{(i)}$ in the end.

In the case of non-deterministic switchover times, we focus on the steady-state probabilities $\nu_n^{(i)}$ that there are $n$ messages waiting at the other station upon exit from station~$i$.
We obtain
\begin{equation*}
\begin{split}
 \nu_n^{(1)} = & \; \nu_0^{(2)} \sum_{j=0}^n \left[\int_0^{\infty} e^{-(\lambda_1+\lambda_2) x} \frac{(\lambda_1 x)^{0}}{0!} \frac{(\lambda_2 x)^{j}}{j!} \diff F_{R_2}(x) \right. \\
 & \qquad \qquad \left( \int_0^{T_1} \int_0^{\infty} e^{-\lambda_2(x+y)}  \frac{(\lambda_2(x+y))^{n-j}}{(n-j)!} g_1(x)  \diff x \, \lambda_1 e^{-\lambda_1 y} \diff y \right. \\
 &  \qquad \qquad \qquad \qquad \qquad \qquad \qquad \qquad \quad + \left. \left. e^{-\lambda_2 T_1} \frac{(\lambda_2 T_1)^{n-j}}{(n-j)!} e^{-\lambda_1 T_1} \right) \right] \\
 & +\sum_{k=0}^\infty \nu_k^{(2)} \sum_{\substack{m=0 \\ m+k \neq 0}}^\infty \sum_{j=0}^n \left[\int_0^{\infty} e^{-(\lambda_1+\lambda_2) x} \frac{(\lambda_1 x)^{m}}{m!} \frac{(\lambda_2 x)^{j}}{j!} \diff F_{R_2}(x) \right. \\ 
 & \qquad \qquad \qquad \qquad \qquad \qquad \left. \int_0^{\infty} e^{-\lambda_2 x} \frac{(\lambda_2 x)^{n-j}}{(n-j)!} g_1^{(*(k+m))}(x) \diff x \right]
\end{split}
\end{equation*}
and get two systems of equations $\nu^{(1)} = \tilde{A} \nu^{(2)}$ and $\nu^{(2)} = \tilde{B} \nu^{(1)}.$
Finally, we can compute $\pi_n^{(i)}$ as mentioned in~\eqref{eq:pi_n} for Strategy~II.

\paragraph{The expected waiting time $f_i$.}
Let $E_i$ be an exponentially distributed random variable with intensity $\lambda_i$ which represents the interarrival time of messages at station~$i$.
We denote by $\min(E_i,T_i)$ the random length of a waiting period at station~$i$.
The timer at station~$i$ is activated if and only if the server finds this station empty upon arrival.
Therefore, we can conclude
$$
f_i = \pi_0^{(i)} \E\left[\min(E_i,T_i) \right] = \frac{\pi_0^{(i)}}{\lambda_i} \left(1-e^{-\lambda_i T_i} \right)
$$
for the expected waiting time at station~$i$ per cycle in steady state.

\paragraph{The expected backward recurrence time $w_i$.}
The quantity $w_i$ equals the expected residual time of a waiting period and is given by
$$
w_i = \frac{ \E\left[\min(E_i,T_i)^2 \right]}{2 \E\left[\min(E_i,T_i) \right]} = \frac{1}{\lambda_i} - \frac{T_i}{e^{\lambda_i T_i}-1}.
$$

\paragraph{The conditional mean switchover time $\tilde{r}_i$.}
If the switchover time is deterministic, we just have $\tilde{r}_i=r_i$ (cf.\ Remark~\ref{rem:r_cond}).
Now, we focus on a non-deterministic switchover time:
Similar but easier than for Strategy~II, the quantity $\tilde{r}_2$ is just the mean switchover time given that there is no arrival at station~$1$ while switching to this station.
We get
\begin{equation}\label{eq:r_IV}
\tilde{r}_2 = \E[R_2 \mid B_0] = \frac{\int_0^{\infty} x e^{-\lambda_1 x} \diff F_{R_2}(x)}{\int_0^{\infty} e^{-\lambda_1 x} \diff F_{R_2}(x)}
\end{equation}
and we can represent $\tilde{r}_1$ in an analogous manner.

\subsection{Proofs of the `worth-waiting' results}
\subsubsection{Preparations}
First, we state two facts which we use later to prove that it is worth waiting with Strategy~II if it is worth waiting with Strategy~IV.
Lemma~\ref{lem:r_cond} concerns an estimate for the mean switchover time given a certain number of message arrivals while switching.

\begin{lem}\label{lem:r_cond}
There is a positive constant $\alpha$ such that
$$
\E[R_2 \mid B_j] \leq \alpha \left(j^2+1 \right)
$$
for all $j \in \mathbb{N}$ with the notation from~\eqref{eq:Bj}.
\end{lem}
\begin{proof}[Sketch of proof]
We recall
$$
\E[R_2 \mid B_j] = \frac{\E_R \left[R^{j+1} e^{-R} \right]}{\lambda_1 \E_R \left[R^j e^{-R} \right]}
$$
for $R \coloneqq \lambda_1 R_2$ from~\eqref{eq:r_cond_Bj} and we introduce the random variable $X$ by
$$
\E\left[f(X) \right] \coloneqq \frac{\E_R \left[f(R) e^{-R} \right]}{\E_R \left[e^{-R} \right]}, \quad f \in C_b.
$$
Then, $X$ has some finite exponential moment and one can show by elementary calculations that there is an $\alpha > 0$ such that
$$
\frac{\E\left[X^{j+1} \right]}{\E\left[X^j \right]} \leq \lambda_1 \alpha \left(j^2+1 \right)
$$
for all $j \in \mathbb{N} $. This finishes the proof.
\end{proof}

The next Lemma~\ref{lem:pi_0} captures the fact that if there may be an additional waiting time due to a larger wait-and-see parameter $\tilde{T}_1 \geq T_1$, rather more messages arrive per cycle.
Therefore, the probability of finding an empty queue upon arrival at station~$1$ becomes smaller.

\begin{lem}\label{lem:pi_0}
Consider a polling model with $N=2$ stations, Strategy~II and $T_2=0$.
Given a $\bar{T}_1>0$, we have
$$
\pi_{\rm{inf}} \coloneqq \inf_{T_1 \in [0,\bar{T}_1]} \pi_0^{(1)}(T_1) > 0.
$$
\end{lem}
\begin{proof}[Sketch of proof]
We can construct an appropriate coupling of two processes representing the polling models with wait-and-see parameter~$T_1$ and~$\tilde{T}_1$ for $0 \leq T_1 \leq \tilde{T}_1$.
Due to the construction, the queue length at station~$1$ upon exit from station~$2$ is always larger for the process with~$\tilde{T}_1$ instead of~$T_1$.
Combining this observation and the ergodic theorem for Markov chains, we obtain
$$
\nu_0^{(2)}(T_1) \geq \nu_0^{(2)} (\tilde{T}_1).
$$
This inequality is equivalent to
$$
\pi_0^{(1)}(T_1)\geq \pi_0^{(1)} (\tilde{T}_1)
$$
due to~\eqref{eq:pi_n}.
Then, we get $\pi_{\rm{inf}} = \pi_0^{(1)} (\bar{T}_1)$.
\end{proof}

We make use of Theorem~\ref{thm:main_2} to prove whether it is worth waiting.
For the purpose of comparison, we recall the formula
$$
\bar{D}^{\rm{exh}} = \frac{\sum_{i=1}^2 \lambda_i b_i^{(2)}}{2(1-\rho_0)} + \frac{r_0 \rho_1 \rho_2}{\rho_0(1-\rho_0)} + \frac{r_0^{(2)}}{2 r_0}
$$
for the mean average queueing delay of a message in a polling model with the exhaustive strategy from~\eqref{eq:D} by setting $f_1=f_2=0$.
Thus, we can rearrange~\eqref{eq:D} into $\bar{D} = \bar{D}^{\rm{exh}} + \Delta \bar{D}$ with
\begin{equation}\label{eq:delta_D}
\begin{split}
 \Delta \bar{D} \coloneqq & -\frac{r_0^{(2)}}{2 r_0} + \frac{r_0^{(2)}}{2 (r_0+f_0)} \\
 & + \frac{\rho_2 f_1}{\rho_0(r_0+f_0)} (r_1 + \tilde{r}_2 + w_1) \\
 & + \frac{\rho_1 f_2}{\rho_0(r_0+f_0)} (\tilde{r}_1 + r_2 + w_2).  
\end{split}
\end{equation}

\subsubsection{\texorpdfstring{Proof of Theorem~\ref{thm:cond_2}}{Proof of Theorem 3}}
Due to $T_2=0$, we have $f_2=0$ and the last line in~\eqref{eq:delta_D} vanishes.
It is worth waiting at station~$1$ if and only if there is a positive parameter of the wait-and-see strategy such that $\Delta \bar{D} <0$.
Since the expected waiting time at station~$1$ equals the total expected waiting time per cycle ($f_1=f_0$), we rearrange inequality $\Delta \bar{D} < 0$ into
$$
\frac{1}{r_0+f_1} \left[\frac{r_0^{(2)}}{2} + \frac{\rho_2}{\rho_0} f_1 \left(r_1 + \tilde{r}_2 + w_1 \right)\right] < \frac{r_0^{(2)}}{2 r_0}
$$
whose validity is equivalent to
\begin{equation}\label{eq:cond_simple_1}
\left[-\frac{r_0^{(2)}}{2r_0} + \frac{\rho_2}{\rho_0} \left(r_1 + \tilde{r}_2 + w_1 \right)\right] f_1 < 0.
\end{equation}
We recall that $w_i$ and $f_i$ are non-negative quantities.
Moreover, we observe that $f_i>0$ holds for all $T_i>0$.
This can be argued by using the expected sojourn times for Strategy~III and by using the steady-state probabilities for Strategy~II and~IV.

\paragraph{Strategy III.}
Note that we have $r_1 + \tilde{r}_2 = r_0$ according to Remark~\ref{rem:r_cond}.
For all $T_1>0$, we see from~\eqref{eq:wi_III} that $w_1$ is greater than the expected length of the first busy period at station~$1$, i.e., there is a function $\Delta_1(T_1) > 0$ such that
$$
w_1 = (r_0+c_2)\frac{\rho_1}{1-\rho_1} + \Delta_1(T_1). 
$$
We insert this representation of $w_1$ into~\eqref{eq:cond_simple_1}, make use of~\eqref{eq:c2_III} and obtain that~\eqref{eq:cond_simple_1} is equivalent to
\begin{equation}\label{eq:cond_III_long}
-\frac{r_0^{(2)}}{2r_0} + \frac{\rho_2}{\rho_0} \left(\frac{1-\rho_2}{1-\rho_0}r_0 + \frac{\rho_1 \rho_2}{1-\rho_0} \left(T_1 + \frac{q_1(T_1)b_1}{1-\rho_1} \right) + \Delta_1(T_1) \right) < 0.
\end{equation}
Because of the property that both functions $\Delta_1(T_1)$ and $q_1(T_1)$ converge to zero for $T_1 \to 0$, we find the sufficient condition
\begin{equation}\label{eq:cond_III}
\frac{r_0^{(2)}}{2r_0^2} - \frac{\rho_2(1-\rho_2)}{\rho_0(1-\rho_0)} > 0
\end{equation}
for `it is worth waiting at station~$1$'.
In order to establish the necessity of this condition, we argue in the following way:
If we assume that~\eqref{eq:cond_III} does not hold, inequality~\eqref{eq:cond_III_long} is not satisfied for all $T_1>0$ because $\Delta_1(T_1)$ and $q_1(T_1)$ are non-negative, and we see that it is not worth waiting at station~$1$.

\paragraph{Strategy IV.}
The difference to Strategy~III is the fact that $w_1 $ does not have to be greater than the expected length of the first busy period at station~$1$.
We just focus on 
\begin{equation}\label{eq:cond_simple_2}
-\frac{r_0^{(2)}}{2r_0} + \frac{\rho_2}{\rho_0} \left(r_1 + \tilde{r}_2 + w_1 \right) < 0
\end{equation}
from~\eqref{eq:cond_simple_1} and observe the property $w_1 \leq T_1$ because a waiting period ends at the latest when the timer expires.
In the same manner as above, we get the necessary and sufficient condition
$$
\frac{r_0^{(2)}}{2r_0 \left(r_1 + \tilde{r}^{\rm{IV}}_2 \right)} - \frac{\rho_2}{\rho_0} > 0
$$
for `it is worth waiting at station~$1$' with $\tilde{r}^{\rm{IV}}_2$ given by~\eqref{eq:r_IV}.
In the case of deterministic switchover times, we just replace $r_0^{(2)}$ by $r_0^2$ and $\tilde{r}^{\rm{IV}}_2$ by $r_2$.

\paragraph{Strategy II.}
We focus again on~\eqref{eq:cond_simple_2} as with Strategy~IV, and $w_1 \leq T_1$ holds since waiting periods can only happen within the minimum sojourn time~$T_1$.
Differently from Strategy~IV, the conditional mean switchover time~$\tilde{r}^{\rm{II}}_2$ depends on the parameter~$T_1$.

First, we prove that it is worth waiting with Strategy~IV if it is worth waiting with Strategy~II.
Therefore, we assume that there is a $T_1>0$ such that~\eqref{eq:cond_simple_2} holds for Strategy~II.
We have to conclude that~\eqref{eq:cond_IV} is satisfied which can be easily seen if we have $\tilde{r}^{\rm{IV}}_2 \leq \tilde{r}^{\rm{II}}_2(T_1)$ for all $T_1>0$.
We continue with proving this inequality.
We recall
$$
\tilde{r}^{\rm{II}}_2 = \sum_{k=0}^\infty p_k^{(1)} \sum_{j=0}^k \frac{\P(A_{k-j}) \P(B_j)}{\P(C_k)} \, \E[R_2 \mid B_j]
$$
from~\eqref{eq:r_II} and~\eqref{eq:r_cond_Ck}, and
$$
\tilde{r}^{\rm{IV}}_2 = \E[R_2 \mid B_0]
$$
from~\eqref{eq:r_IV}.
We use the representation of $\E[R_2 \mid B_j]$ from~\eqref{eq:r_cond_Bj} and the Cauchy-Schwarz inequality to get
$$
\E[R_2 \mid B_j] \leq \E[R_2 \mid B_{j+1}]
$$
for all $j \in \mathbb{N}_0$.
This property suffices in order to conclude $\tilde{r}^{\rm{IV}}_2 \leq \tilde{r}^{\rm{II}}_2(T_1)$ for all~$T_1>0$.

Next, we have to prove that it is worth waiting with Strategy~II if it is worth waiting with Strategy~IV.
Let~\eqref{eq:cond_IV} be satisfied, i.e., there is a $T^{\rm{IV}}_1 > 0$ such that~\eqref{eq:cond_simple_2} holds for $\tilde{r}^{\rm{IV}}_2$ and $w_1^{\rm{IV}}(T^{\rm{IV}}_1)$.
We are done if there is a $T_1>0$ such that 
$$ 
\tilde{r}^{\rm{II}}_2(T_1) + w_1^{\rm{II}}(T_1) \leq \tilde{r}^{\rm{IV}}_2 + w_1^{\rm{IV}}(T^{\rm{IV}}_1)
$$
because~\eqref{eq:cond_simple_2} is the criterion for `it is worth waiting with Strategy~II' as well.
We observe 
\begin{equation*}
\begin{split}
 \tilde{r}^{\rm{II}}_2 & = p_0^{(1)} \E[R_2 \mid B_0] + \sum_{k=1}^\infty p_k^{(1)} \sum_{j=0}^k \frac{\P(A_{k-j}) \P(B_j)}{\P(C_k)} \, \E[R_2 \mid B_j] \\
 & \leq \E[R_2 \mid B_0] + \sum_{k=1}^\infty p_k^{(1)} \E[R_2 \mid B_k]
\end{split}
\end{equation*}
and define $\eps \coloneqq \frac{w_1^{\rm{IV}}(T^{\rm{IV}}_1)}{2}$.
Due to $ \tilde{r}^{\rm{IV}}_2 = \E[R_2 \mid B_0]$ and $w_1^{\rm{II}}(T_1) \leq T_1$,
it suffices to show that there is a positive $T_1 < \eps$ such that
$$
\sum_{k=1}^\infty p_k^{(1)} \E[R_2 \mid B_k] < \eps.
$$
We recall
$$
p_k^{(1)} = \frac{\pi_k^{(1)} \int_0^{T_1} P_{k,0}(x) \diff x}{\sum_{l=0}^\infty \pi_l^{(1)} \int_0^{T_1} P_{l,0}(x) \diff x}
$$
from~\eqref{eq:p_k}.
First, we estimate the quantity $\int_0^{T_1} P_{k,0}(x) \diff x$ that is the expected length of the total waiting time during the stay at station~$1$ given that there are $k$ messages waiting upon arrival.
We get
\begin{equation*}
\begin{split}
 \int_0^{T_1} P_{0,0}(x) \diff x & \geq T_1 \, \P(\text{no message arrives at station~$1$ within the time $T_1$}) \\
 & = T_1 e^{-\lambda_1 T_1}
\end{split}
\end{equation*}
and
\begin{equation*}
\begin{split}
 \int_0^{T_1} P_{k,0}(x) \diff x & \leq T_1 \, \P(\text{the length of the first busy period} \leq T_1) \\
 & \leq T_1 \, \P(\text{the sum of $k$ independent service times} \leq T_1) \\
 & \leq T_1 \left(1- e^{-\mu_1 T_1} \sum_{j=0}^{k-1} \frac{(\mu_1 T_1)^j}{j!} \right) \\
 & = T_1 e^{-\mu_1 T_1} \left(e^{\mu_1 T_1} - \sum_{j=0}^{k-1} \frac{(\mu_1 T_1)^j}{j!} \right) \\
 & = T_1 e^{-\mu_1 T_1} \sum_{j=k}^{\infty} \frac{(\mu_1 T_1)^j}{j!} \\
 & = T_1 e^{-\mu_1 T_1} (\mu_1 T_1)^{k} \sum_{j=0}^{\infty} \frac{(\mu_1 T_1)^{j}}{(j+k) \cdots (j+1)j!} \\
 & \leq T_1 (\mu_1 T_1)^{k}
\end{split}
\end{equation*}
for all $k \in \mathbb{N}$ where we use the Erlang($k,\mu_1$) distribution function in the third line.
Now, we can bound $p_k^{(1)}$ for all $k \in \mathbb{N}$ from above by
$$
p_k^{(1)} \leq \frac{T_1 (\mu_1 T_1)^{k}}{\pi_0^{(1)} T_1 e^{-\lambda_1 T_1} } = \frac{e^{\lambda_1 T_1}}{\pi_0^{(1)}} (\mu_1 T_1)^{k}.
$$
Using Lemmas~\ref{lem:r_cond} and~\ref{lem:pi_0} with $\bar{T}_1 \coloneqq \frac{1}{\mu_1}$ in the first two lines and using limits of geometric series, we obtain for $T_1 \in \left(0, \bar{T}_1 \right)$ with $q \coloneqq \mu_1 T_1 < 1$
\begin{equation*}
\begin{split}
 \sum_{k=1}^\infty p_k^{(1)} \E[R_2 \mid B_k] & \leq \sum_{k=1}^\infty \frac{e^{\lambda_1 T_1} }{\pi_0^{(1)}} (\mu_1 T_1)^{k} \alpha \left(k^2 + 1 \right) \\
 & \leq \frac{e^{\frac{\lambda_1}{\mu_1}}}{\pi_{\rm{inf}}} \alpha \left(\sum_{k=1}^\infty  k^2 q^{k} + \sum_{k=1}^\infty q^{k} \right) \\
 & = \frac{e^{\frac{\lambda_1}{\mu_1}} }{\pi_{\rm{inf}}} \alpha \left(\frac{q(1+q)}{(1-q)^3} + \frac{q}{1-q} \right).
\end{split}
\end{equation*}
Finally, we are done because the term in the last line converges to zero for $T_1 \to 0$.
\hspace*{\fill}$\square$

\subsubsection{\texorpdfstring{Proof of Theorem~\ref{thm:cond_2_sym}}{Proof of Theorem 5}}
We focus on inequality $\Delta \bar{D} <0$ which can be rearranged into
$$
\frac{1}{r_0+f_0} \left[\frac{r_0^{(2)}}{2} + \frac{\rho_2}{\rho_0} f_1 \left(r_1 + \tilde{r}_2 + w_1 \right) + \frac{\rho_1}{\rho_0} f_2 \left(\tilde{r}_1 + r_2 + w_2 \right)\right] < \frac{r_0^{(2)}}{2 r_0}.
$$

\paragraph{Strategy III.}
We can proceed in an analogous manner as in the proof of Theorem~\ref{thm:cond_2}.
Using the symmetry $\rho_1 = \rho_2$, we obtain the necessary and sufficient condition
$$
\frac{r_0^{(2)}}{r_0^2} - \frac{1-\rho_1}{1-\rho_0} >  0
$$
for `it is worth waiting' at both stations with $T_1=T_2>0$.

\paragraph{Strategy II and IV.}
In addition to the procedure in the proofs above, we have to extend Lemma~\ref{lem:pi_0} by setting $T_1=T_2>0$.
Then, for a totally symmetric polling model we get the necessary and sufficient condition
\begin{equation}\label{eq:cond_IV_sym}
\frac{r_0^{(2)}}{r_0 \left(r_1 + \tilde{r}^{\rm{IV}}_2 \right)} > 1 
\end{equation}
for `it is worth waiting' at both stations in the same way.
A short calculation shows that $\tilde{r}^{\rm{IV}}_2 \leq r_2$ holds.
Therefore, we can conclude that~\eqref{eq:cond_IV_sym} is satisfied if and only if the switchover times are non-deterministic.
\hspace*{\fill}$\square$

\subsubsection{\texorpdfstring{Proof of Corollary~\ref{cor:conseq}}{Proof of Corollary 6}}
We just have to set the system parameters such that the condition (inequality) in Theorem~\ref{thm:cond_2} or~\ref{thm:cond_2_sym} is fulfilled for Strategy~II and~IV but not for Strategy~III.
\hspace*{\fill}$\square$

\paragraph{Acknowledgement.}
The work of S.\ Schwinn is supported by the `Excellence Initiative' of the German Federal and State Governments via the Graduate School of Computational Engineering at Technische Universit\"at Darmstadt.

\end{document}